\documentclass[letterpaper, 10 pt, conference]{eccconf}
\IEEEoverridecommandlockouts
\usepackage{cite}
\usepackage{amsmath,amssymb,amsfonts}
\usepackage{mathtools}
\usepackage{algorithmic}
\usepackage{graphicx, subfigure}	
\usepackage{textcomp}
\usepackage{xcolor}
\usepackage{lipsum}
\usepackage{siunitx}
\usepackage[colorlinks]{hyperref}

\def\BibTeX{{\rm B\kern-.05em{\sc i\kern-.025em b}\kern-.08em
    T\kern-.1667em\lower.7ex\hbox{E}\kern-.125emX}}

\graphicspath{{figures/}}
\newtheorem{lemma}{Lemma}
\newtheorem{definition}{Definition}
\newtheorem{assumption}{Assumption}
\newtheorem{proposition}{Proposition}
\newtheorem{theorem}{Theorem}
\newtheorem{remark}{Remark}

\begin{document}

\title{Novel Multi-objective Switched Model Predictive Control with Feasibility and Stability Guarantees}

\author{Elias Niepötter, Adrian Grimm, and Torbjørn Cunis
\thanks{This work was not supported by any organization.}
\thanks{The authors are with the Institute of Flight Mechanics and Controls, University of Stuttgart, 70569 Stuttgart.
{\tt\small \{ adrian.grimm | tcunis \}@ifr.uni-stuttgart.de}}%
}


\maketitle
\renewcommand{\thefootnote}{\fnsymbol{footnote}}

\begin{abstract}
As the relevance of control systems capable of dealing with multiple objectives rises (e.g. being economic while maintaining a certain performance), multi-objective Switched Model Predictive Control combines all the advantages of Model Predictive Control while dealing with multiple objectives. We propose two novel frameworks, a nominal and a robust framework to guarantee recursive feasibility of each Model Predictive Controller under arbitrary switching and assure asymptotic stability of the closed-loop system applying the nominal framework and Input-to-State stability using the robust framework. The presented frameworks employ methods from switched systems, enabling the utilization of a supervisor control instance which allows for complex objectives and multi-objective control. Our numerical example confirms the superior performance of our proposed frameworks compared to a standard Model Predictive Control approach.
\end{abstract}



\section{Introduction}
Model Predictive Control (MPC) is one of the few control techniques capable of handling nonlinearities, multivariable systems, and state and input constraints combined in a rigorous mathematical framework. Switched Systems, a class of hybrid systems, are of increasing relevance due to the need for control systems capable of handling increasingly complex systems. However, switching between different MPCs is currently not extensively researched, although there are several advantages of such a method. This work proposes a parallel-Switched Model Predictive Control (pSMPC) scheme, which allows for multi-objective optimization, a complex high-level cost function, and mode decoupling with constraints. The increasing relevance of control systems handling multiple objectives at the same time is shown in \cite{peitz_mooc}. Here, the attribute parallel means that all MPCs are solved in parallel and all solutions are available at each time step. This approach has not been covered in the literature so far. The scheme consists of multiple MPC running in parallel with additional constraints to guarantee recursive feasibility under arbitrary switching signals. We present both a nominal and robust framework. A \textit{supervisor} coordinates the switching and decides which control input is forwarded to the plant based on the information generated by each MPC. The supervisor incorporates switching constraints and a high-level objective function.

\subsection{Literature Review}
The works in \cite{franco_practically, franco_multiple, parisini_hybrid_rhc}
investigates switching between different prediction models and controllers in a hybrid system setup, but focuses on approximated receding horizon controllers through neural networks. The authors do not address closed-loop stability or recursive feasibility, e.g., through stabilizing terminal ingredients. Multiple Model Predictive Control (MMPC) \cite{du_mmpc, kuure_mmpc, yang_mmpc} utilizes multiple linear prediction models to control a nonlinear plant. The consequences of switching frequently back and forth between different controllers, such as potential destabilization (see e.g. \cite{liberzon_switched}), are not studied. Instead, these works focus on selecting a suitable model for a specific operating region. Soft switching MPC techniques like \cite{mehdizadeh_soft, hao_soft} employ multiple linear prediction models for a nonlinear system and switch 'softly' between the controllers to deal with nonlinearities. This approach extends the MMPC framework in such a way that the objective functions of the MPCs are combined through a convex combination before and after switching to ensure a smooth transition. However, it focuses on linear MPC only and no guarantees for recursive feasibility of the presented algorithms are given. Works on switched MPC for switched systems \cite{yuan_smpc_ssys, zhang_smpc_ssys, zhuang_smpc_ssys} focus on switched systems as a plant model. These works present stability and feasibility characterizations, but are mainly based on so-called {\em dwell time} arguments which might be suboptimal regarding performance.

\subsection{Contribution}
The main contribution of this work is the derivation of two novel multi-objective MPC frameworks with rigorous stability and feasibility guarantees. To some extent, the proposed methods resemble a combination of switched systems and quasi-infinite horizon MPC with terminal ingredients. We employ methods from switched systems to design a multi-objective MPC framework. To the best of the authors' knowledge, no existing method in the literature currently provides such a framework in combination with stability and feasibility guarantees. We propose stabilizing constraints under arbitrary switching signals and provide constructive algorithms to apply those constraints in a practical manner. In addition to a nominal framework, a robust framework is proposed which guarantees Input-to-State stability (ISS) in the presence of large disturbances. We prove stability and recursive feasibility for both frameworks, filling a gap in the literature. 

The remainder of the paper is organized as follows:
Section~\ref{sec:prelim} presents the control problem as well as the relevant background of the utilized methodologies. The need for new methods and the novel MPC schemes are derived in Section~\ref{sec:main}. A numerical example demonstrates the additional advantages of the proposed methods in Section~\ref{sec:numex}.


\section{Preliminaries} \label{sec:prelim}
\textbf{Notation.} The energy norm of a vector $\boldsymbol{x} \in \mathbb R^n$ w.r.t. to a matrix $\boldsymbol{A} \in \mathbb R^{n \times n}$ is indicated by $||\boldsymbol{x}||_{A} = \sqrt{\boldsymbol{x}^T\boldsymbol{A}\boldsymbol{x}}$. The Euclidean norm is denoted by $||\cdot||_2$ and $||\cdot||_\infty$ denotes the infinity norm of a signal. The absolute value of a scalar is indicated by $|\cdot|$. The notation $\boldsymbol{A}_{[a,b]}$ indicates that from the matrix $\boldsymbol{A}$ the columns $a$ to $b$ are selected. A composition of function $f$ and $g$ is denoted by $f \circ g$. We use the comparison function classes $\mathcal K$, $\mathcal K_\infty$, and $\mathcal {KL}$ in this work, see \cite{kellet_ccf} for details.

\subsection{Control Structure}
This work utilizes a control structure inspired by \cite{parisini_hybrid_rhc}.
The plant dynamics and the discrete prediction models are given by
\begin{align}
    \dot{\boldsymbol{\xi}}(t) &= F(\boldsymbol{\xi}(t),\boldsymbol{\mu}(t)) + \boldsymbol{\nu}(t) \label{eq:plant} \\
    \boldsymbol{x}^{i}(k+1) &= f_{i}(\boldsymbol{x}^{i}(k),\boldsymbol{u}^{i}(k)) \label{eq:predictionModel}
\end{align}
where $\boldsymbol{\xi}(t) \in \mathcal{X} \subseteq \mathbb{R}^{n}$ is the state vector, 
$\boldsymbol{\mu}(t) \in \mathcal{U} \subseteq \mathbb{R}^{m}$ is the input vector and
$\boldsymbol{\nu}(t) \in \mathcal{D} \subseteq \mathbb{R}^{n}$ is some disturbance vector. It is assumed that the states are measurable.
Each MPC is formulated in discrete time with prediction models $f_i$, where the index $i \in \mathcal{I} = \left[ 1 \dots M_c \right]$ indicates a specific prediction model. The state and input vectors are $\boldsymbol{x}^i \in \mathcal{X}_i \subseteq \mathbb{R}^{n_i}$ and $\boldsymbol{u}^i \in \mathcal{U}_i \subseteq \mathbb{R}^{m}$ where $\boldsymbol{x}^{i}$ and $\boldsymbol{u}^{i}$ indicate the state and input vector of the $i$-th MPC and $\mathcal{X}_i$ and $\mathcal{U}_i$ are compact for all $i \in \mathcal I$. The index set $\mathcal{I}$ is minimal, that is, no two controllers $(i,j) \in \mathcal{I}$ with $i \neq j$ are identical. A \textit{Translator} connects the continuous-time plant to the discrete-time MPCs and the \textit{Supervisor}. In this work, the Translator samples the state signal $\boldsymbol{\xi}(\cdot)$ at times $t_k$, $k \in \mathbb N$, and maps it to the prediction states via $\boldsymbol{x}^i(k) = T^i(\boldsymbol{\xi}(t_k))$. The Supervisor receives the current state of the plant and a set of information from each MPC, namely the optimal input and state trajectories, as well as the optimal value $V_i$. The supervisor decides which control input to feed-forward to the plant based on the super objective $\mathcal{J}$. The control structure is shown in Fig.~\ref{fig:control_structure}.

\begin{figure}[h]
    \begin{center}
        \includegraphics[width=1\linewidth]{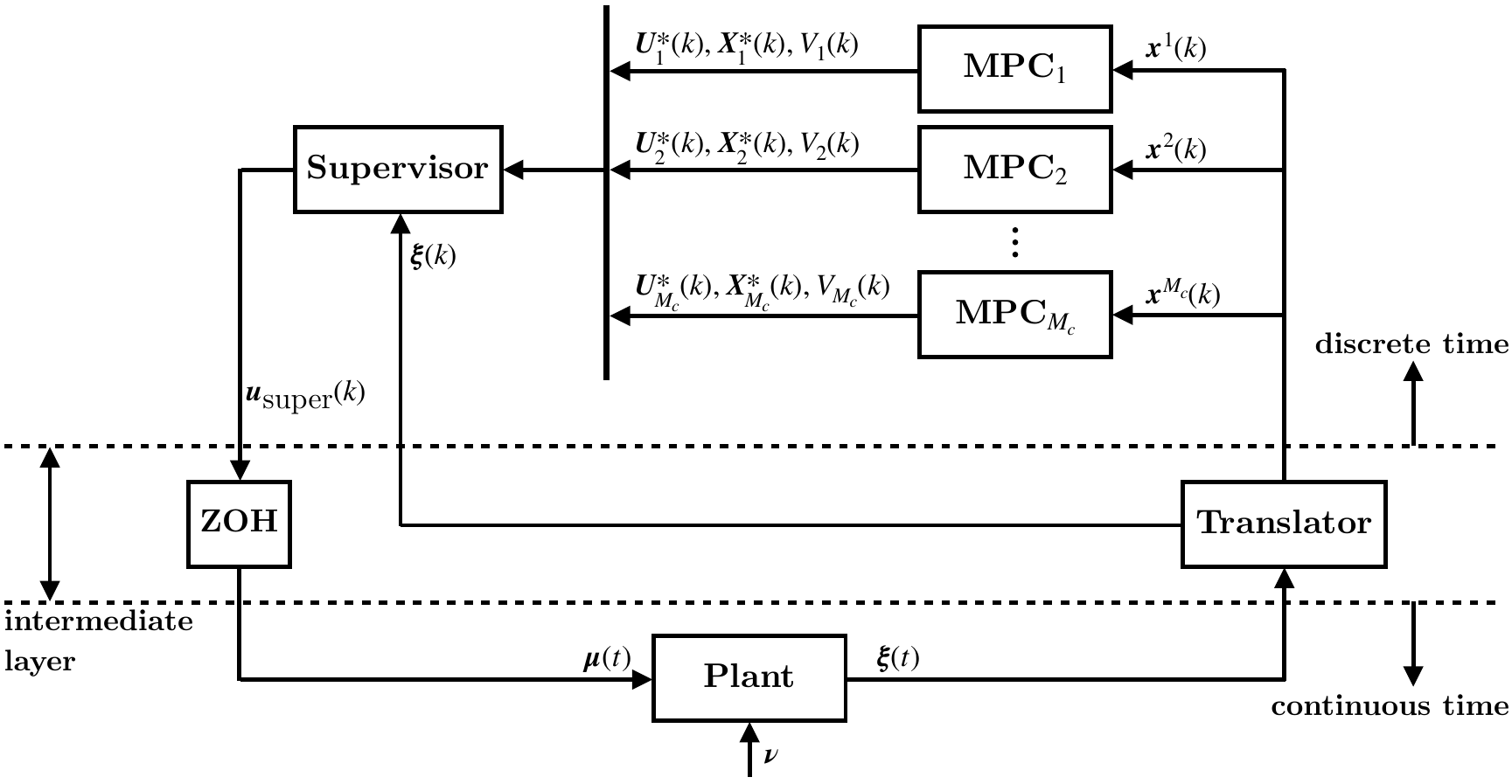}
        \vspace{-2em}\caption{Overview of the control structure inspired by \cite{parisini_hybrid_rhc}}
        \label{fig:control_structure}
    \end{center}
\end{figure}

\begin{assumption}
    \label{asm:sampledDataSystem}
    The discrete-time predicted signals satisfy $\boldsymbol{x}^i(k+1) = T^i(\boldsymbol{\xi}(t_{k+1}))$ under ZOH and sample times ${t_k}$, ${k \in N}$.
\end{assumption}
\begin{assumption}
\label{asm:regularity}
    The dynamics $F$ are Lipschitz continuous in their first argument and continuous in the second. The dynamics $f_i$ are continuously differentiable over $\mathcal{X}$ and $\mathcal{U}$.
\end{assumption}
\begin{assumption}
\label{asm:boundedNorm}
    There exist class $\mathcal{K}_{\infty}$ function $\overline{\alpha}$ s.t.
    \begin{equation}
        \sum\limits_{i = 1}^{M_c} ||\boldsymbol{x}^{i}(k)||_2 \geq \overline{\alpha}(||\boldsymbol{\xi}(t_k)||_2)
    \end{equation}
    holds for all sample times $t_k$, $k \in \mathbb{N}$.
\end{assumption}
Assumptions \ref{asm:sampledDataSystem} and \ref{asm:regularity} are standard for sampled-data systems and Model Predictive Control. Assumption \ref{asm:boundedNorm} ensures regularity of the mappings $T^{i}$.

\subsection{Nominal and Robust Model Predictive Control}
The nominal MPC scheme is the discrete Quasi-Infinite Horizon MPC from \cite{rajhans}. Let $k$ be the current discrete time instance and $\boldsymbol{x}^{i}(k)$ the state of the $i$-th subsystem at time instance $k$. Let $\boldsymbol{x}^{i}(j \mid k)$ and $\boldsymbol{u}^{i}(j \mid k)$ denote the predicted state and input at time instance $j$ where $j$ is the time step of the prediction. Let $\boldsymbol{U}_i(k) = \left[ \boldsymbol{u}^{i}(k \mid k), \dots, \boldsymbol{u}^{i}(k+N-1 \mid k) \right]$ denote a candidate input sequence for the interval $\left[ k, k+N \right]$. Then the $i$-th optimal control problem (OCP) is given by
\begin{equation}
    \begin{aligned}
        \label{eq:ocp}
        V_i(k) = & \min_{\boldsymbol{U}_i(k)} \quad J^{i}(\boldsymbol{x}^{i}(k), \boldsymbol{U}_i(k))\\
        \text { s.t. } \quad &\boldsymbol{x}^{i}(j+1 \mid k) = f_i(\boldsymbol{x}^{i}(j \mid k),\boldsymbol{u}^{i}(j \mid k)) \\
        &\boldsymbol{x}^{i}(j \mid k) \in \mathcal{X}_i \text{ , } \boldsymbol{u}^{i}(j \mid k) \in \mathcal{U}_i \\
        &\boldsymbol{x}^{i}(k+N \mid k) \in \mathcal{X}^{N}_{i} \text{ , } \boldsymbol{x}^{i}(k \mid k) = \boldsymbol{x}^{i}(k)
    \end{aligned}
\end{equation}
for all $j \in [k, N-1]$,
where the cost functional is given by
\begin{equation}
\begin{aligned}
    J^{i}(\boldsymbol{x}^{i}(k), & \boldsymbol{U}_i(k)) = ||\boldsymbol{x}^{i}(k+N \mid k)||^2_{\boldsymbol{P}_i} \\
    & + \sum_{j=k}^{k+N-1} ||\boldsymbol{x}^{i}(j \mid k)||^2_{\boldsymbol{Q}_i} + ||\boldsymbol{u}^{i}(j \mid k)||^2_{\boldsymbol{R}_i}.
\end{aligned}
\end{equation}
The weighting matrices $\boldsymbol{Q}_i$ and $\boldsymbol{R}_i$ are positive definite (p.d.) for all $i$ and the length of the prediction horizon is given by $N$. The sets $\mathcal{X}_i$, $\mathcal{U}_i$, and $\mathcal{X}^{N}_{i}$ are polytopes. Let $\boldsymbol{U}^{*}_i(k) = \left[ \boldsymbol{u}^{*i}(k \mid k), \dots, \boldsymbol{u}^{*i}(k+N-1 \mid k) \right]$ denote the optimal solution of \eqref{eq:ocp} at $k$. The first optimal input is applied to the plant via zero-order hold, i.e., 
\begin{equation}
\label{eq:rhc}
\boldsymbol{\mu}(t) = \boldsymbol{u}^{*i}(k \mid k), \quad t \in (t_k, t_{k+1}].
\end{equation}
Equivalently, $\boldsymbol{X}^{*}_i(k) = \left[ \boldsymbol{x}^{i}(k \mid k), \dots, \boldsymbol{x}^{*i}(k+N \mid k) \right]$ denotes the optimal state trajectory at time instance $k$. The terminal ingredients $\mathcal{X}^{N}_{i}$, $\boldsymbol{P}_i$ and the terminal linear controller $\boldsymbol{K}_i$ are designed so that the requirements of the QIH-MPC scheme are met. Then, system \eqref{eq:plant} with $\boldsymbol{\nu}(t)=0$ implementing \eqref{eq:rhc} and OCP \eqref{eq:ocp}
is asymptotically stable. The optimal value function $V_i(k)$ is a Lyapunov function. This is proven, e.g. in \cite{rajhans}. Continuing with the preliminaries for the robust MPC, the following ISS statements are notation-wise adopted from \cite{jiang_iss, kohler_robust}. Some nonlinear closed-loop prediction dynamics with disturbance
\begin{equation}
    \label{eq:issSystem}
    \tilde{\boldsymbol{x}}(k+1) = z(\tilde{\boldsymbol{x}}(k),\boldsymbol{\nu}(k)) 
\end{equation}
are ISS if there exists a class $\mathcal{KL}$ function $\beta$ and class $\mathcal{K}$ function $\gamma$ satisfying
\begin{equation}
    \label{eq:iss}
    ||\tilde{\boldsymbol{x}}(k,\tilde{\boldsymbol{x}}_0,\boldsymbol{\nu})||_2 \leq \beta(||\tilde{\boldsymbol{x}}_0||_2,k) + \gamma(||\nu||_\infty)
\end{equation}
which holds for all $\nu \in \mathbb{R}^n$ and all initial conditions $\tilde{\boldsymbol{x}}_0 \in \mathbb{R}^n$ where $\tilde{\boldsymbol{x}}(k,\tilde{\boldsymbol{x}}_0,\boldsymbol{\nu})$ denotes the trajectory of $\tilde{\boldsymbol{x}}(k)$ starting at $\tilde{\boldsymbol{x}}_0$ with disturbance $\boldsymbol{\nu}$. Note that the disturbance is not predicted. A system which is ISS satisfies
\begin{equation}
\label{eq:ag}
\varlimsup_{k \rightarrow+\infty}||\tilde{\boldsymbol{x}}\left(k, \tilde{\boldsymbol{x}}_{0}, \boldsymbol{\nu}\right)||_2 \leq \gamma\left(\|\boldsymbol{\nu}\|_{\infty}\right) \quad \forall \tilde{\boldsymbol{x}}_{0}, \boldsymbol{\nu}(\cdot)
\end{equation}
for some $\gamma \in \mathcal{K}_{\infty}$, which is also called the \textit{asymptotic gain property}
(\textit{AG}). A continuous function $V$ is called an ISS-Lyapunov function if there exist class $\mathcal K_\infty$ functions $\alpha_1, \alpha_2, \alpha_3, \gamma$ such that
\begin{align}
    \alpha_1(||\tilde{\boldsymbol{x}}||_2) &\leq V(\tilde{\boldsymbol{x}}) \leq \alpha_2(||\tilde{\boldsymbol{x}}||_2) \label{eq:issV1}\\
    V(f(\tilde{\boldsymbol{x}}, \boldsymbol{\nu}))-V(\tilde{\boldsymbol{x}}) &\leq -\alpha_3(||\tilde{\boldsymbol{x}}||_2)+\gamma(||\boldsymbol{\nu}||_2), \label{eq:issV2}
\end{align}
holds for all $\tilde{\boldsymbol{x}} ,\nu \in \mathbb{R}^n$.
\begin{assumption}
    \label{asm:openloopAS}
    The plant $F$ is open-loop asymptotically stable.
\end{assumption}
Assumption \ref{asm:openloopAS} is made in \cite{kohler_robust} for the derivation of the robust stability properties of the following MPC scheme. In this work, this Assumption is only made for the robust framework. The robust MPC scheme proposed by \cite{kohler_robust} is given by the OCP
\begin{equation}
    \begin{aligned}
        \label{eq:rocp}
        \overline{V}_i(k) = &\min _{\boldsymbol{U}_i(k)} J^{i}(\boldsymbol{\overline{x}}^{i}(k), \boldsymbol{U}_i(k)) + \lambda_i V_{\delta,i}(\boldsymbol{x}^{i}(k),\boldsymbol{\overline{x}}^{i}(k)) \\
        \text { s.t. } \quad & \boldsymbol{x}^{i}(j+1 \mid k)= f_i(\boldsymbol{x}^{i}(j \mid k),\boldsymbol{u}^{i}(j \mid k))\\
        & \boldsymbol{x}^{i}(j \mid k) \in \mathcal{X}_i \text{ , } \boldsymbol{u}^{i}(j \mid k) \in \mathcal{U}_i \\
        & \boldsymbol{x}^{i}(k+N \mid k) \in \mathcal{X}^{N}_{i} \text{ , } \boldsymbol{x}^{i}(k \mid k) = \boldsymbol{\overline{x}}^{i}(k).
    \end{aligned}
\end{equation}
The main difference to \eqref{eq:ocp} is that the measurement $\boldsymbol{x}^{i}(k)$ of the plant is a hard constraint on the initial state in \eqref{eq:ocp} and a soft constraint with a penalty term $\lambda_i V_{\delta,i}(\boldsymbol{x}^{i}(k),\boldsymbol{\overline{x}}^{i}(k))$ in \eqref{eq:rocp}. Thus, the resulting OCP \eqref{eq:rocp} resembles a projection of \eqref{eq:ocp} onto its set of feasible initial conditions and is therefore always feasible. Assuming again the existence of terminal ingredients, system \eqref{eq:plant} implementing \eqref{eq:rhc} and OCP \eqref{eq:rocp} is ISS. The optimal value function $\overline{V}_i(k)$ is an ISS-Lyapunov function. The proof can be found in \cite{kohler_robust}.

\subsection{Switched Systems Stability}
Let $\Xi_i \coloneqq \left\lbrace \tilde{t}^{i}_0, \tilde{t}^{i}_1, \dots, \tilde{t}^{i}_{2r}, \tilde{t}^{i}_{2r+1}, \dots,  \right\rbrace$
be a sequence of switching instances $\tilde{t}^{i} \in \mathbb R$. By convention, even indices will denote an activation of subsystem and odd indices a deactivation. The switched system is described by
\begin{equation}
\label{eq:switchedSystem}
\dot{\boldsymbol{x}}^{\sigma}(t)=F_{\sigma(t)}(\boldsymbol{x}^{\sigma}(t)) + \boldsymbol{\nu}(t)
\end{equation}
where $\sigma: \mathbb R \to \mathcal I$ is the switching signal mapping and $\boldsymbol{\nu}$ is a potential disturbance.\cite{liberzon_switched}

\begin{lemma}
\label{lm:MLF}
\textit{Let $V_i$ denote a Lyapunov function corresponding to the $i$-th system of the family \eqref{eq:switchedSystem} where $\boldsymbol{\nu} = 0$. Assume that there exist p.d. functions
$W_i: \mathbb R^{n_i} \to \mathbb R$ for all $i \in \mathcal{I}$. The switched system \eqref{eq:switchedSystem} is globally asymptotically stable if
\begin{equation}
    \label{eq:MLF}
    V_i\left(\boldsymbol{x}^{i}\left(\tilde{t}^{i}_{2r+2}\right)\right) - V_i\left(\boldsymbol{x}^{i}\left(\tilde{t}^{i}_{2r}\right)\right)
    \leq -W_i\left(\boldsymbol{x}^{i}\left(\tilde{t}^{i}_{2r}\right)\right)
\end{equation}
for all $i \in \mathcal{I}$.}
\end{lemma}
The proof of Lemma \ref{lm:MLF} can be found \cite{liberzon_switched}. In the context of the control structure presented, the supervisor defines the switching sequence of each subsystem $\Xi_i$ according to its objective $\mathcal{J}$. That way, the supervisor realizes the switching signal $\sigma$ and ensures \eqref{eq:MLF}.


\section{Main Results} \label{sec:main}
The following notion of a \textit{parallel-Switched MPC} will be used subsequently.
\begin{definition}
\label{def:smpc}
\textit{A parallel-Switched MPC describes a set of Model Predictive Controllers running in parallel and a single plant forming a nonlinear switched system equal to \eqref{eq:switchedSystem}.}
\end{definition}
In the context of this work, a switch means that the supervisor decides which closed-loop system (MPC + plant) is active.

\subsection{Nominal: Leader-Restrictor Framework}
In the following, the nominal framework assumes $\boldsymbol{\nu}(t)=0$. One of the advantages of using an MPC is that a Lyapunov function of the closed-loop system is easily accessible. Hence, one can use Lemma \ref{lm:MLF} to guarantee stability by adopting the switching signal accordingly. Unfortunately, the standard MPC scheme \eqref{eq:ocp} produces a \textit{prediction mismatch} whenever the MPC is switched. This mismatch hinders a direct application of \eqref{eq:MLF} to guarantee stability by virtue of Lemma~\ref{lm:MLF}. The feasibility guarantee is lost because the it is not guaranteed that the initial state of the next time step is feasible for all MPCs.

\begin{definition}
\label{def:predictionMismatch}
\textit{A prediction mismatch for the $i$-th MPC is present if the measurement at time
instance $k$ does not match the first predicted state of the $i$-th MPC at the previous time instance $k-1$, s.t. $\boldsymbol{x}^{i}(k) \neq \boldsymbol{X}^{*}_{i,[1,1]}(k-1)$.}
\end{definition}

\begin{proposition}
\label{prop:predictionMismatchOccurence}
\textit{For nominal parallel-Switched MPC implementing \eqref{eq:ocp} each, a prediction mismatch cannot be ruled out when a switch occurs.}
\end{proposition}

\begin{proof}
When initializing all MPCs with the same initial state measurement $\boldsymbol{x}(k)$, one MPC $i$ will be selected to be active, realizing the first predicted state in $\boldsymbol{X}^{*}_{i,[1,1]}(k)$. As the index set $\mathcal{I}$ is minimal, it is not guaranteed that all MPCs will have the same minimizer $\boldsymbol{U}^{i*}(k)$. Consequently, a prediction mismatch cannot be ruled out.
\end{proof}

In standard QIH-MPC, a prediction mismatch is ruled out via standard assumptions (no disturbance, no uncertainties, ...). That way, feasibility is guaranteed, see, e.g. \cite{rajhans} for details. To resolve this problem for nominal parallel-Switched MPC, the \textit{Leader-Restrictor} framework is introduced. Here, the OCP \eqref{eq:ocp} is augmented to
\begin{subequations}
    \label{eq:cOCP}
\begin{align}
    \hat{V}_i = &\min_{\boldsymbol{U}_i(k)} \quad J^{i}(\boldsymbol{x}^{i}(k), \boldsymbol{U}_i(k)) \\
\begin{split}
    \label{eq:cocp_qih}
    \text{s.t.} \quad & \boldsymbol{x}^{i}(j+1 \mid k) = f_i(\boldsymbol{x}^{i}(j \mid k),\boldsymbol{u}^{i}(j \mid k)) \\
    & \boldsymbol{x}^{i}(j \mid k) \in \mathcal{X}_i, \quad \boldsymbol{u}(j \mid k) \in \mathcal{U}_i \\
    & \boldsymbol{x}^{i}(k+N \mid k) \in \mathcal{X}^{N}_{i}, \quad \boldsymbol{x}^{i}(k \mid k) = \boldsymbol{x}^{i}(k)
\end{split} \\
\begin{split}
    \label{eq:cocp_dynamicComposition}
    \text{and} \quad & \boldsymbol{x}^{l}(s+1 \mid k) = f_l(\boldsymbol{x}^{l}(s \mid k),\boldsymbol{u}^{i}(s \mid k)) \\
    & \boldsymbol{x}^{l}(s \mid k) \in \mathcal{X}_l, \quad \boldsymbol{u}^{i}(s \mid k) \in \mathcal{U}_l \\
    & \boldsymbol{x}^{l}(k+N \mid k) \in \mathcal{X}^{N}_{l}, \\
    & \boldsymbol{x}^{l}(k+1 \mid k) = \boldsymbol{x}^{i}(k+1 \mid k)
    \end{split}
\end{align}
\end{subequations}
for all $j \in [k, k+N-1]$ and $s \in [k+1, k+N-1]$.
We call the OCP \eqref{eq:cOCP} \textit{composed}, as it composes all dynamics and constraints of $\mathcal{I}$. The $i$-th MPC, as it \textit{leads} the optimization, will be called the \textit{leader}. The constraint set of the $l$-th MPC, as it only further \textit{restricts} the solution, will be called the \textit{restrictor}. The following additional assumptions are made:
\begin{assumption}
    \label{asm:terminalAssumption}
    For all $i,l \in \mathcal{I}$, the terminal set of the $l$-th system $\mathcal{X}^N_l$ is invariant under action of the $i$-th terminal controller $\boldsymbol{K}_i$ and the terminal control input of the $i$-th system satisfies $-\boldsymbol{K}_i\boldsymbol{x}^{i} \in \mathcal{U}_l$.
\end{assumption}
\begin{assumption}
    \label{asm:feasibilityFirstCall}
    OCP \eqref{eq:cOCP} is feasible for $\boldsymbol{x}^i(0)$.
\end{assumption}


\begin{lemma}
\label{lm:cocp}
\textit{A nominal parallel-Switched MPC implementing \eqref{eq:cOCP} is recursively feasible for all switching signals if Assumptions \ref{asm:sampledDataSystem} to \ref{asm:feasibilityFirstCall} hold.}
\end{lemma}

\begin{proof}
The objective function and the first set of constraints \eqref{eq:cocp_qih} are equivalent to a standard QIH-MPC. Following, e.g., \cite{rajhans}, recursive feasibility of the $i$-th MPC of \eqref{eq:cOCP} is guaranteed by the fact that the optimal input sequence $\boldsymbol{U}^{*}_{i}(k)$ drives the system into the terminal set $\mathcal{X}_i^N$ under control of the $i$-th MPC. Consequently, at step $k+1$, the input sequence $\boldsymbol{U}^{\text{feas}}_{i}(k+1) = \left[ \boldsymbol{U}^{*}_{i,[1,N-1]}(k), -K_i \boldsymbol{x}^{i}(N \mid k) \right]$ is feasible at $k+1$ as $\boldsymbol{x}^{i}(k+1) = \boldsymbol{X}^{*}_{i,[1,1]}(k)$ under the first optimal input $\boldsymbol{U}^{*}_{i,[0,0]}(k)$.
Similarly, as $\boldsymbol{U}^{*}_{i}(k)$ creates a trajectory $\boldsymbol{X}_{l}(k)$ starting at $\boldsymbol{X}^{*}_{i,[1,1]}(k)$ that satisfies the respective set of constraints
\eqref{eq:cocp_dynamicComposition}, the input sequence
$\boldsymbol{U}^{\text{feas}}_{l}(k+1) = \left[ \boldsymbol{U}^{*}_{i,[1,N-1]}(k), -K_l \boldsymbol{x}^{l}(N \mid k) \right]$ is also feasible w.r.t. to the $l$-th controller.
\end{proof}

\begin{remark}
Without Assumption \ref{asm:terminalAssumption}, it is not guaranteed that $\boldsymbol{U}^{\text{feas}}_{i}(k+1)$ is feasible for the set of constraints \eqref{eq:cocp_dynamicComposition}.
\end{remark}

\begin{remark}
    The optimal value function $\hat{V}_i$ is a Lyapunov function for the closed loop of \eqref{eq:plant} and the $i$-th MPC. Standard arguments for QIH-MPC hold, see, e.g., \cite{rajhans} for details.
\end{remark}

\begin{theorem}
\label{thm:nominalSMPC}
Let $W_i: \mathbb R^{n_i} \to \mathbb R$ be a p.d. function for all $i \in \mathcal{I}$. The parallel-Switched MPC is asymptotically stable provided that the supervisor constraints the switching to
\begin{equation}
    \label{eq:switchingMLFcomposed}
    \hat V_i\left(\boldsymbol{x}^{i}\left(\tilde{t}^{i}_{2r+2}\right)\right) - \hat V_i\left(\boldsymbol{x}^{i}\left(\tilde{t}^{i}_{2r}\right)\right)
    \leq -W_i\left(\boldsymbol{x}^{i}\left(\tilde{t}^{i}_{2r}\right)\right)
\end{equation}
for all $i \in \mathcal{I}$.
\end{theorem}

\begin{proof}
Recall that $\hat{V}_i$ is a Lyapunov function for the $i$-th system of the family \eqref{eq:switchedSystem} which implements \eqref{eq:cOCP} each. By virtue of Lemma \ref{lm:cocp}, recursive feasibility of the parallel-Switched MPC is always guaranteed, enabling any switching signal. Hence, \eqref{eq:switchingMLFcomposed} assures asymptotic stability by virtue of Lemma \ref{lm:MLF}.
\end{proof}

\subsection{Robust: Multiple ISS-Lyapunov Functions}
We now introduce an alternative to Theorem \ref{thm:nominalSMPC} for systems with disturbances which guarantees recursive feasibility and ISS. Regarding ISS, the following result is an extension of Lemma \ref{lm:MLF} with $\boldsymbol{\nu} \neq 0$.
\begin{assumption}
    \label{asm:knownDisturbance}
    The disturbance $\boldsymbol{\nu} \in \mathcal{D}$ is known at time $t$.
\end{assumption}
\begin{theorem}
\label{thm:robustSMPC}
\textit{Let $W_i: \mathbb R^{n_i} \to \mathbb R$ and $\gamma_i$ be class $\mathcal{K}_{\infty}$ functions for all $i \in \mathcal{I}$. The parallel-Switched MPC is ISS w.r.t. the disturbance $\boldsymbol{\nu}$ provided that the supervisor constraints the switching to
\begin{multline}
    \label{eq:ISS-MLF}
    \overline V_i\left(\boldsymbol{x}^{i}\left(\tilde{t}^{i}_{2r+2}\right)\right) - \overline V_i\left(\boldsymbol{x}^{i}\left(\tilde{t}^{i}_{2r}\right)\right) \\
    \leq -W_i\left(\boldsymbol{x}^{i}\left(\tilde{t}^{i}_{2r}\right)\right) + \gamma_i(||\boldsymbol{\nu}(\tilde t^i_{2r+2})||_2)
\end{multline}
for all $i \in \mathcal{I}$.}
\end{theorem}

\begin{proof}
Note that the optimal value function $\overline V_i$ is an ISS-Lyapunov function corresponding to the $i$-th system of the family \eqref{eq:switchedSystem} which implements \eqref{eq:rocp} each and that recursive feasibility is guaranteed by the design of \eqref{eq:rocp}. In order to show that the switched system \eqref{eq:switchedSystem} with OCP \eqref{eq:rocp} and constraint \eqref{eq:ISS-MLF} is indeed ISS, the 0-GAS property and AG is needed \cite{tran_equivalence}. If $\boldsymbol{\nu} = 0$, inequality \eqref{eq:ISS-MLF} becomes \eqref{eq:MLF} since $\overline{V}_i$ are (ISS)-Lyapunov functions according to \cite{kohler_robust}, hence the pSMPC is 0-GAS. In the presence of a disturbance signal $\boldsymbol{\nu}(\cdot)$ recall from \eqref{eq:issV1}
\begin{equation}
    V_i(\boldsymbol{x}^{i}) \leq \alpha_2 \left(||\boldsymbol{x}||_2\right)
\end{equation}
and since each subsystem in $\mathcal{I}$ is AG
\begin{equation}
    \varlimsup_{k \rightarrow + \infty} || \boldsymbol{x}(k,\boldsymbol{x}_0,\boldsymbol{\nu}) ||_2 \leq \gamma_i(||\boldsymbol{\nu}||_{\infty})
\end{equation}
we can state the following bound:
\begin{equation}
\label{eq:limitV}
\varlimsup_{r \rightarrow \infty} V_i\left(\boldsymbol{x}^{i}\left(\tilde{t}^{i}_{2r}\right)\right) \leq \left( \alpha_{2,i} \circ \gamma_i \right)(||\boldsymbol{\nu}||_{\infty})
\end{equation}
Since $\alpha_{i,2},\gamma_i \in \mathcal{K}_{\infty}$, define $\delta_i \coloneqq \left( \alpha_{2,i} \circ \gamma_i \right) \in \mathcal{K}_{\infty}$ for all $i \in \mathcal{I}$.
Consequently, we can bound the limit of the left hand side (LHS) of \eqref{eq:ISS-MLF} as
\begin{equation}
\label{eq:limitLHS}
\varlimsup_{r \rightarrow \infty} \left| V_i\left(\boldsymbol{x}^{i}\left(\tilde{t}^{i}_{2r+2}\right)\right) - V_i\left(\boldsymbol{x}^{i}\left(\tilde{t}^{i}_{2r}\right)\right) \right| \leq \delta_i(||\boldsymbol{\nu}||_{\infty}).
\end{equation}
Hence, at each activation instance (assuming a large enough $t \geq 0$), Eq.~\eqref{eq:limitLHS} holds as each ISS-Lyapunov function is naturally bounded by zero from below. Using the lower limit of \eqref{eq:limitLHS} yields
\begin{equation}
    -\delta_i(||\boldsymbol{\nu}||_{\infty}) \leq \varliminf_{r \rightarrow \infty} \left[ -W_i\left(\boldsymbol{x}^{i}\left(\tilde{t}^{i}_{2r}\right)\right) + \gamma_i(||\boldsymbol{\nu}(\tilde t^i_{2r+2})||_2) \right].
\end{equation}
As
\begin{multline}
    -W_i\left(\boldsymbol{x}^{i}\left(\tilde{t}^{i}_{2r}\right)\right) + \gamma_i(||\boldsymbol{\nu}(\tilde t^i_{2r+2})||_2) \\
    \leq  -W_i\left(\boldsymbol{x}^{i}\left(\tilde{t}^{i}_{2r}\right)\right) + \gamma_i\left(\|\boldsymbol{\nu}\|_{\infty}\right) 
\end{multline}
it follows that
\begin{align}
    -\delta_i(||\boldsymbol{\nu}||_{\infty}) &\leq \varliminf_{r \rightarrow \infty} \left[ -W_i\left(\boldsymbol{x}^{i}\left(\tilde{t}^{i}_{2r}\right)\right) \right] + \gamma_i\left(\|\boldsymbol{\nu}\|_{\infty}\right) \\
    &\leq - \varlimsup_{r \rightarrow \infty} \left[ W_i\left(\boldsymbol{x}^{i}\left(\tilde{t}^{i}_{2r}\right)\right) \right] + \gamma_i\left(\|\boldsymbol{\nu}\|_{\infty}\right) 
\end{align}
and finally
\begin{equation}
\label{eq:finalUpperBound}
    \delta_i(||\boldsymbol{\nu}||_{\infty}) + \gamma_i\left(\|\boldsymbol{\nu}\|_{\infty}\right) \geq \lim_{r \rightarrow \infty} \left[ W_i\left(\boldsymbol{x}^{i}\left(\tilde{t}^{i}_{2r}\right)\right) \right].
\end{equation}
Define $\Phi_i \coloneqq \delta_i + \gamma_i \in \mathcal{K}_\infty$ for all $i \in \mathcal{I}$. 
Then
\begin{align}
    \varlimsup_{r \rightarrow \infty} \left|\left|\boldsymbol{x}^{i}\left(\tilde{t}^{i}_{2r}\right)\right|\right|_{2} \leq (W_i^{-1} \circ \Phi_i) \left(\|\boldsymbol{\nu}\|_{\infty}\right)
\end{align}
where $W_i^{-1}$ and $W_i^{-1} \circ \Phi_i$ are again class $\mathcal K_\infty$ functions. Finally, let $\Psi \left(\|\boldsymbol{\nu}\|_{\infty}\right) \coloneqq \max\limits_{i \in \mathcal{I}} (W_i^{-1}\circ\Phi_i)\left(\|\boldsymbol{\nu}\|_{\infty}\right)$, with $\Psi \in K_\infty$, then
\begin{equation}
    \varlimsup_{r \rightarrow \infty} \left|\left|\boldsymbol{x}^{i}\left(\tilde{t}^{i}_{2r}\right)\right|\right|_{2} \leq \Psi\left(\|\boldsymbol{\nu}\|_{\infty}\right)
\end{equation}
for all $i\in \mathcal{I}$. Consequently, all subsystems are is ISS w.r.t. the disturbance $\boldsymbol{\nu}$. Moreover, because of Assumption \ref{asm:boundedNorm}, it follows
\begin{equation}
    \varlimsup_{r \rightarrow \infty} \overline{\alpha}\left(\left|\left|\boldsymbol{\xi}\left(\tilde{t}^{i}_{2r}\right)\right|\right|_2\right) \leq M_c \Psi\left(\|\boldsymbol{\nu}\|_{\infty}\right)
\end{equation}
which yields
\begin{equation}
    \varlimsup_{r \rightarrow \infty} \left|\left|\boldsymbol{\xi}\left(\tilde{t}^{i}_{2r}\right)\right|\right|_{2} \leq \Omega\left(\|\boldsymbol{\nu}\|_{\infty}\right)
\end{equation}
where $\Omega \coloneqq \left( \overline{\alpha}^{-1} \circ M_c \Psi \right) \in \mathcal{K}_\infty$.
Hence, the plant is ISS w.r.t. the disturbance $\boldsymbol{\nu}$.
\end{proof}

\subsection{Discussion}
In this work, we switch between multiple MPC to solve a multiple-objective optimal control task. As long as a single MPC is active, its respective optimal value function decreases in the sense of an (ISS)-Lyapunov function until the supervisor decides to switch to a different MPC. In the nominal framework, switching is constrained by the strict dissipativity condition \eqref{eq:switchingMLFcomposed}. However, in any applied control system disturbances will occur, hence \eqref{eq:switchingMLFcomposed} might be too restrictive for switching in practical application. In the robust framework, on the other hand, the disturbance term $\gamma_i(\cdot)$ weakens the dissipativity constraint and hence facilitates switching under disturbances. However, due to the prediction mismatch, switching from a state that is infeasible for the MPC to be switched to may also lead to a prohibitive increase of the optimal value function in \eqref{eq:ISS-MLF}. One can further weaken the dissipativity constraint by modelling the prediction mismatch as a disturbance. In this case, and if switching occurs for $t \rightarrow \infty$, the system will only be practically stable. Lastly, note that the functions $W_i$ and $\gamma_i$ in \eqref{eq:switchingMLFcomposed} and \eqref{eq:ISS-MLF} are design parameters, as the switching signal is not known a priori.


\section{Numerical Example} \label{sec:numex}
We demonstrate the effectiveness of Theorem \ref{thm:nominalSMPC} and \ref{thm:robustSMPC} applied to a real world system. The comparison without disturbance is referred to as `nominal case', whereas the comparison with disturbance is referred to as `robust case.' 
A linear system is selected to facilitate the synthesis of terminal ingredients. 

\subsection{Setup}
We consider the linearized continuous longitudinal flight dynamics of a generic fixed-wing aircraft. The states $\boldsymbol{x} = \left( \alpha, q, V, \theta \right)$ are the angle of attack $\alpha$, the pitch rate $q$, the velocity $V$ and the pitch angle $\theta$. The input vector is given by $\boldsymbol{u} = \left( \eta, \delta \right)$ with the elevator deflection $\eta$ and the throttle command $\delta$. The system is fully observable and controllable, furthermore, all states are assumed to be measurable. All states and inputs have box constraints with $|\alpha| \leq \SI{10}{\deg}$, $|q| \leq \SI{100}{\deg\per\second}$, $|V| \leq \SI{10}{\meter\per\second}$, $|\theta| \leq \SI{50}{\deg}$, $|\eta| \leq \SI{50}{\deg}$ and $\delta \in [0,1]$. The control task is to stabilize the aircraft from the initial condition $\boldsymbol{x}_0 = \left( -10, 0 , 0, 0 \right)$.
In this example, the pSMPC approach is compared with three single QIH-MPC setups which employ \eqref{eq:ocp} and \eqref{eq:rocp} in the nominal case and robust case, respectively. All MPCs have the same nominal discrete prediction model. The difference between the MPCs is the definition of the weight matrices $\boldsymbol{Q}_i$ and $\boldsymbol{R}_i$ as well as the resulting terminal ingredients. Given the weightings
\begin{equation}
    \boldsymbol{Q}_1=\begin{bmatrix} 1 & 0 & 0 & 0 \\ 0 & 1 & 0 & 0 \\ 0 & 0 & 0.1 & 0 \\ 0 & 0 & 0 & 1 \end{bmatrix} ,\quad
    \boldsymbol{Q}_2=\begin{bmatrix} 5 & 0 & 0 & 0 \\ 0 & 5 & 0 & 0 \\ 0 & 0 & 0.5 & 0 \\ 0 & 0 & 0 & 5 \end{bmatrix}
\end{equation}
the \textbf{Equal-$\mathbf{I}$ MPC} is weighted with $\boldsymbol{Q}_1$ and $\boldsymbol{R} = \mathbb{I}^{2 \times 2}$. The \textbf{High-$\mathbf{Q}$ MPC} has $\boldsymbol{Q}_2$ and $\boldsymbol{R} = \mathbb{I}^{2 \times 2}$. The weights for the \textbf{High-$\mathbf{R}$ MPC} are $\boldsymbol{Q}_1$ and $\boldsymbol{R} = 5\mathbb{I}^{2 \times 2}$. The pSMPC can switch between \textbf{High-$\mathbf{Q}$} and \textbf{High-$\mathbf{R}$ MPC}.

The terminal sets $\mathcal{X}^{N}_{i}$ are selected based on the Maximal Output Admissible Set \cite{moas}. The solutions of Discrete Algebraic Riccati Equations (DARE) are used to compute discrete Linear-Quadratic Regulators $\boldsymbol{K}_i$ and the terminal penalties $\boldsymbol{P}_i$. The additional penalty functions $V_{\delta,i}$ are designed according to \cite[Section~6]{kohler_robust}. 
The resulting optimal control problems in \eqref{eq:rocp} and \eqref{eq:cOCP} are formulated as Quadratic Programs (QP) and solved via CasADi's \cite{casadi} Opti stack with the conic solver qpoases \cite{qpoases}. 
The prediction horizon is $N = 40$ for each MPC. The simulation runs with $\Delta t = \SI{0.01}{\second}$ while the MPCs have a discretization of $\Delta t_{MPC} = \SI{0.05}{\second}$. However, each MPC is called with $\Delta t$. We choose $W(\boldsymbol{x})= \sigma \| \boldsymbol{x} \|_2^2$ with $\sigma = 10^{-1}$ in the nominal case and $\sigma = 10^{-3}$ in the robust case, while $\gamma(||\boldsymbol{\nu}||_2)=10^4||\boldsymbol{\nu}||_2^2$. For each robust MPC, we set $\lambda = 10$. The disturbance is modelled by a Gaussian noise applied to the elevator with expectation $\mu_{\nu} = \SI{0}{\deg}$ and standard deviation $\sigma_{\nu} = \SI{0.5}{\deg}$. The prediction mismatch is added to the disturbance. It is assumed that the disturbance can be observed perfectly.

\subsection{Supervisor Objective}
To demonstrate the advantage of the pSMPC approach, the high-level objective of the supervisor is to minimize a deviation from the total energy of the aircraft at the trim point. The total energy $TE$ is defined as the  sum of the kinetic energy deviation $KE = m/2V^2$ and the potential energy deviation $PE = mg\Delta z$. This objective can be motivated by the fact that the aircraft's trim point is usually optimized w.r.t. fuel consumption. Maintaining the trim speed and trim altitude is therefore crucial for efficiency. The proposed pSMPC framework facilitates such an incorporation of additional control objectives. The supervisor objective $\mathcal{J}$ is thus
\begin{equation}
\label{eq:superObj}
    \mathcal{J} = \sum\limits_{j=0}^{N} \underbrace{\sqrt{KE(j)^2 + PE(j)^2}}_{\Delta E},
\end{equation}
where the sum is calculated with the predicted trajectories of each MPC at each simulation step and supervisor call. The altitude is computed via an integration of the vertical component of the velocity $\dot{z}=\sin(\gamma)V$, where $\gamma = \theta-\alpha$ is the flight path angle. 

\subsection{Results}
Table \ref{tbl:results} shows the numerical values of the supervisor objective for all controllers. The values of $\mathcal{J}$ shown in Table \ref{tbl:results} are computed as a post-processing quantity based on the simulated states of the plant for which \eqref{eq:superObj} is evaluated over the full simulation length of $\SI{5}{\second}$. The pSMPC performs best in both the nominal and robust case.

\vspace{-1em}
\begin{table}[h]
    \centering
    \caption{Summary of results}
    \label{tbl:results}
    \begin{tabular}{|p{1.5cm}|p{2.25cm}|p{2.25cm}|p{0.5cm}|}
        \hline
        \textbf{Controller} & \textbf{Nominal} $\mathcal{J}$ [MJ] &  \textbf{Robust} $\mathcal{J}$ [MJ] & \textbf{ID} \\
        \hline
        Equal-$\mathbf{I}$ & 4.69 & 4.87  & -\\
        \hline
        High-$\mathbf{Q}$ & 7.80 & 7.95  & 1\\
        \hline
        High-$\mathbf{R}$ & 3.79 & 4.23  & 2\\
        \hline
        pSMPC & 3.65 & 4.15  & -\\
        \hline
    \end{tabular}
\end{table}

Figures \ref{fig:nominal_deltaE} and \ref{fig:nominal_switching} show energy deviation and the switching signal as well as the switching permission for the nominal case. The switching permission shows whether Eq. \eqref{eq:switchingMLFcomposed} or \eqref{eq:ISS-MLF} are fulfilled (1) or violated (0). The switching signal shows the active controller where initially a high state penalty and afterwards a high input penalty results in a more efficient control of the aircraft. Regarding the dissipativity constraints, \eqref{eq:switchingMLFcomposed} is always satisfied, indicating strictly decreasing optimal values despite the switch.

\begin{figure}[h]
    \begin{center}
        \includegraphics[width=1\linewidth]{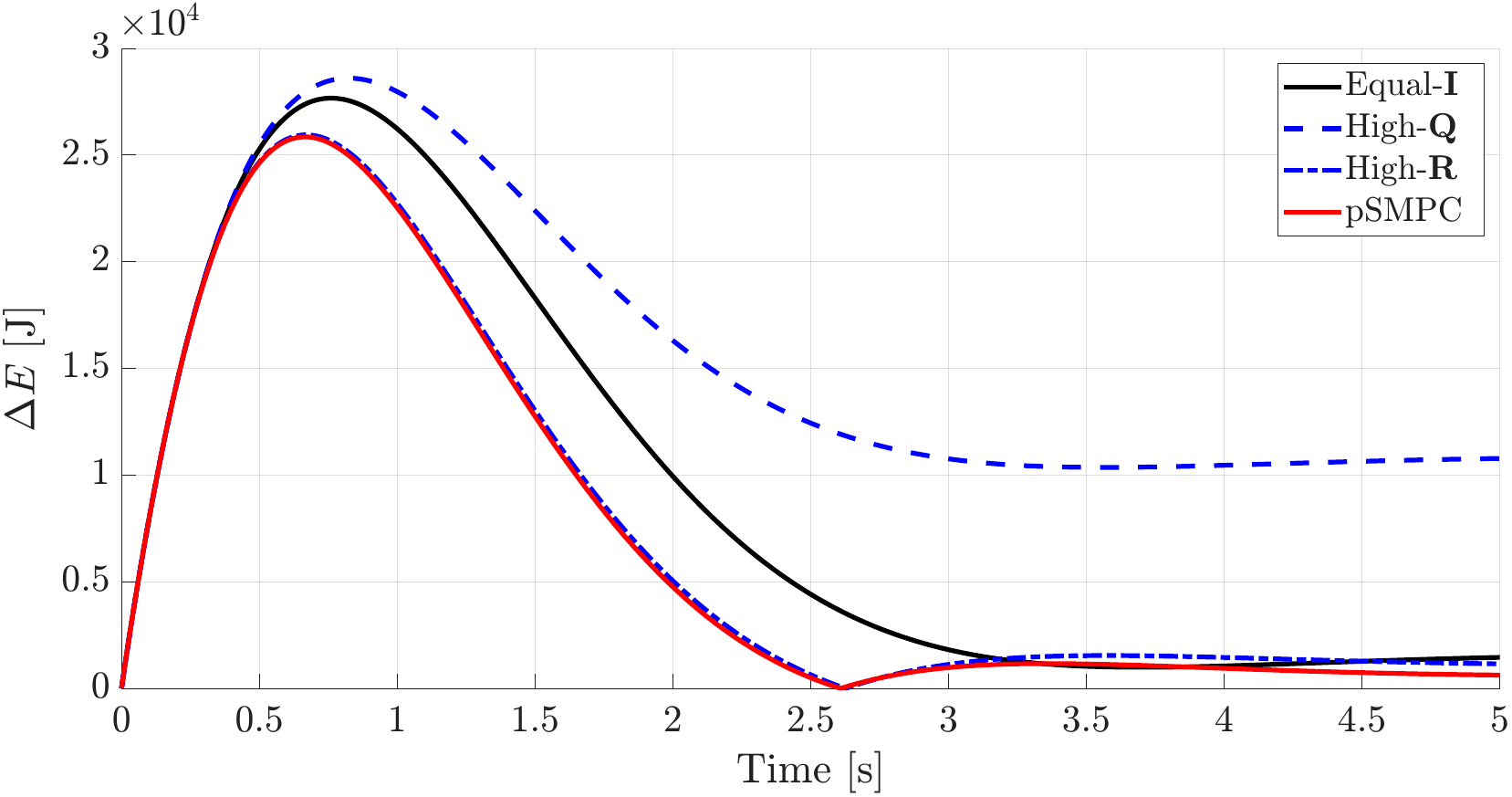}
        \vspace{-2em}\caption{Deviation of the energy over time for all four control schemes (nominal).}
        \label{fig:nominal_deltaE}
    \end{center}
\end{figure}

\begin{figure}[h]
    \begin{center}
        \includegraphics[width=1\linewidth]{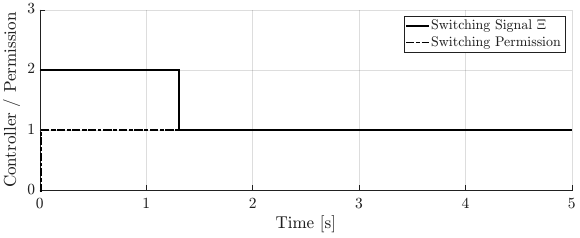}
        \vspace{-2em}\caption{Switching signal and permission (nominal).}
        \label{fig:nominal_switching}
    \end{center}
\end{figure}

The results in the robust case are very similar despite two additional switches between both controllers shown in \ref{fig:robust_switching}.

\begin{figure}[h!]
    \begin{center}
        \includegraphics[width=1\linewidth]{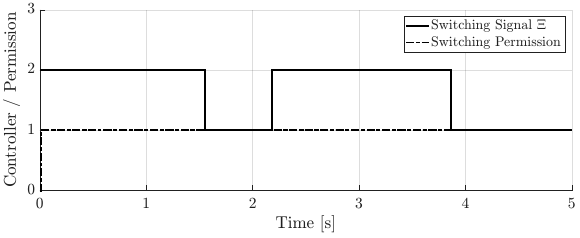}
        \vspace{-2em}\caption{Switching signal and permission (robust).}
        \label{fig:robust_switching}
    \end{center}
\end{figure}


\section{Conclusion}
This work addresses the problem of nonlinear control design under multiple objectives. We propose novel approaches for multi-objective MPC by introducing a parallel-Switched Model Predictive Control scheme with formal guarantees. Two frameworks are presented, imposing nominal and robust stability as well as recursive feasibility through switching conditions for a supervisory unit. A numerical example demonstrates the effectiveness of our approaches. Furthermore, we see connections to \textit{distributed} and \textit{network} control as well as multi-objective optimization. Future developments may go in the direction of exploiting potential synergies with other MPC and switching schemes. Furthermore, computing the invariant terminal set of the nominal pSMPC  is left to future work.


\begin{thebibliography}{00}
\bibitem{casadi}                    J. A. E. Andersson, J. Gillis, G. Horn, J. B. Rawlings, und M. Diehl, "CasADi: a software framework for nonlinear optimization and optimal control", Math. Prog. Comp., Bd. 11, Nr. 1, S. 1–36, März 2019, doi: 10.1007/s12532-018-0139-4.
\bibitem{du_mmpc}                   J. Du, J. Chen, J. Li, und T. A. Johansen, "Multiple Model Predictive Control for Nonlinear Systems Based on Self-Balanced Multimodel Decomposition", Ind. Eng. Chem. Res., Bd. 61, Nr. 1, S. 487–501, Jan. 2022, doi: 10.1021/acs.iecr.1c02426.
\bibitem{qpoases}                   H. J. Ferreau, C. Kirches, A. Potschka, H. G. Bock, und M. Diehl, "qpOASES: a parametric active-set algorithm for quadratic programming", Math. Prog. Comp., Bd. 6, Nr. 4, S. 327–363, Dez. 2014, doi: 10.1007/s12532-014-0071-1.
\bibitem{franco_practically}        E. Franco, S. Sacone, und T. Parisini, "Practically stable nonlinear receding-horizon control of multi-model systems", in 2004 43rd IEEE Conference on Decision and Control (CDC) (IEEE Cat. No.04CH37601), Nassau, Bahamas: IEEE, 2004, S. 3241-3246 Vol.3. doi: 10.1109/CDC.2004.1428973.
\bibitem{franco_multiple}           E. Franco, S. Sacone, und T. Parisini, "Stable multi-model switching control of a class of nonlinear systems", in Proceedings of the 2004 American Control Conference, Boston, MA, USA: IEEE, 2004, S. 1873–1878 Bd.2. doi: 10.23919/ACC.2004.1386853.
\bibitem{moas}                      E. G. Gilbert und K. T. Tan, "Linear systems with state and control constraints: the theory and application of maximal output admissible sets", IEEE Trans. Automat. Contr., Bd. 36, Nr. 9, S. 1008–1020, Sep. 1991, doi: 10.1109/9.83532.
\bibitem{hao_soft}                  G.-C. Hao, Y.-L. Zhang, S.-X. Wen, X. Du, B. Yang, und X.-M. Sun, "A Softly Switching Multiple Model Predictive Control for Aero-engines", IFAC-PapersOnLine, Bd. 54, Nr. 10, S. 477–482, 2021, doi: 10.1016/j.ifacol.2021.10.208.
\bibitem{jiang_iss}                 Z.-P. Jiang, E. Sontag, und Y. Wang, "Input-to-State Stability for discrete-time nonlinear systems".
\bibitem{kuure_mmpc}                M. Kuure-Kinsey und B. W. Bequette, "Multiple Model Predictive Control: A State Estimation based Approach", in 2007 American Control Conference, New York, NY, USA: IEEE, Juli 2007, S. 3739–3744. doi: 10.1109/ACC.2007.4283005.
\bibitem{kohler_robust}             J. Köhler und M. N. Zeilinger, "A model predictive control framework with robust stability guarantees under large disturbances", 21. Juni 2025, arXiv: arXiv:2207.10216. doi: 10.48550/arXiv.2207.10216.
\bibitem{liberzon_switched}         D. Liberzon, Switching in Systems and Control. in Systems \& Control: Foundations \& Applications. Boston, MA: Birkhäuser, 2003. doi: 10.1007/978-1-4612-0017-8.
\bibitem{mehdizadeh_soft}           B. Mehdizadeh Gavgani, A. Farnam, F. Vanbecelaere, J. D.M. De Kooning, K. Stockman, und G. Crevecoeur, "Soft switching multiple model predictive control with overlapping cross-over time strategy in an industrial high speed pick and place application", Control Engineering Practice, Bd. 144, S. 105813, März 2024, doi: 10.1016/j.conengprac.2023.105813.
\bibitem{parisini_hybrid_rhc}       T. Parisini und S. Sacone, "A hybrid receding-horizon control scheme for nonlinear systems", Systems \& Control Letters, Bd. 38, Nr. 3, S. 187–196, Okt. 1999, doi: 10.1016/S0167-6911(99)00064-X.
\bibitem{rajhans}                   C. Rajhans, S. C. Patwardhan, und H. Pillai, "Discrete time formulation of quasi infinite horizon nonlinear model predictive control scheme with guaranteed stability", IFAC-PapersOnLine, Bd. 50, Nr. 1, S. 7181–7186, Juli 2017, doi: 10.1016/j.ifacol.2017.08.602.
\bibitem{tran_equivalence} D. N. Tran, C. M. Kellett, und P. M. Dower, "Equivalences of Stability Properties for Discrete-Time Nonlinear Systems", IFAC-PapersOnLine, Bd. 48, Nr. 11, S. 772–777, 2015, doi: 10.1016/j.ifacol.2015.09.283.
\bibitem{yang_mmpc}                 Z. Yang, Y. Li, und J. E. Seem, "Multi-model predictive control for wind turbine operation under meandering wake of upstream turbines", Control Engineering Practice, Bd. 45, S. 37–45, Dez. 2015, doi: 10.1016/j.conengprac.2015.08.009.
\bibitem{yuan_smpc_ssys}            C. Yuan, Y. Gu, W. Zeng, und P. Stegagno, "Switching Model Predictive Control of Switched Linear Systems with Average Dwell Time", in 2020 American Control Conference (ACC), Denver, CO, USA: IEEE, Juli 2020, S. 2888–2893. doi: 10.23919/ACC45564.2020.9147362.
\bibitem{zhang_smpc_ssys}           L. Zhang, S. Zhuang, und R. D. Braatz, "Switched model predictive control of switched linear systems: Feasibility, stability and robustness", Automatica, Bd. 67, S. 8–21, Mai 2016, doi: 10.1016/j.automatica.2016.01.010.
\bibitem{zhuang_smpc_ssys}          S. Zhuang, H. Gao, und Y. Shi, "Model Predictive Control of Switched Linear Systems With Persistent Dwell-Time Constraints: Recursive Feasibility and Stability", IEEE Trans. Automat. Contr., Bd. 68, Nr. 12, S. 7887–7894, Dez. 2023, doi: 10.1109/TAC.2023.3248279.
\bibitem{peitz_mooc} S. Peitz and M. Dellnitz, "A Survey of Recent Trends in Multiobjective Optimal Control—Surrogate Models, Feedback Control and Objective Reduction", MCA, vol. 23, no. 2, p. 30, June 2018, doi: 10.3390/mca23020030.

\bibitem{kellet_ccf} C. M. Kellett, "A compendium of comparison function results", Math. Control Signals Syst., vol. 26, no. 3, pp. 339–374, Sept. 2014, doi: 10.1007/s00498-014-0128-8.

\end{thebibliography}
\end{document}